\documentclass{amsart}
\usepackage{amssymb}
\usepackage{a4wide}

\theoremstyle{plain}

\newtheorem*{thm}{Theorem}
\newtheorem*{cor}{Corollary}
\newtheorem*{lem}{Lemma}

\newtheorem*{prop}{Proposition}

\newcommand{\Ann}{\mathrm{Ann}}
\newcommand{\ZZ}{\mathbb{Z}}

\newcommand{\NN}{\mathbb{N}}
\newcommand{\CC}{\mathbb{C}}
\newcommand{\mfm}{\mathfrak{m}}
\newcommand{\mfh}{\mathfrak{h}}
\newcommand{\mfg}{\mathfrak{g}}

\newcommand{\End}[2]{\mathrm{End}_{#1}(#2)}
\newcommand{\Ker}[1]{\mathrm{Ker}(#1)}

\renewcommand{\Im}{\mathrm{Im}}
\title{Injective Modules over Down-up algebras}
\author{Paula A.A.B. Carvalho}
\address{Departamento de Matemática Pura da Faculdade de Ciências da Universidade do Porto, Portugal}
\author{Christian Lomp}
\address{Departamento de Matemática Pura da Faculdade de Ciências da Universidade do Porto, Portugal}
\author{Dilek Pusat-Yilmaz}
\address{Departament of Mathematics, Izmir Institute of Technology, Urla, Turkey}
\thanks{Special thanks to Patrick Smith who suggested the topic at the X. Antalya Algebra Days to the first author.}
\begin{document}

\begin{abstract}
The purpose of this paper is to study finiteness conditions on injective hulls of simple modules over Noetherian Down-Up algebras.
We will show that the Noetherian Down-Up algebras $A(\alpha,\beta,\gamma)$ which are fully bounded are precisely those which are module-finite over a central subalgebra. 
We show that injective hulls of simple $A(\alpha,\beta,\gamma)$-modules are locally Artinian provided the roots of $X^2-\alpha X - \beta$ are distinct roots of unity or both equal to one. \\ \begin{center}{\it dedicated to Patrick F. Smith - teacher and friend}\end{center}
\end{abstract}
\maketitle
\section{Injective hulls of simple modules over noetherian rings}
Injective modules are the building blocks in the theory of Noetherian rings. Matlis showed that any indecomposable injective module
over a commutative Noetherian ring is isomorphic to the injective hull $E(R/P)$ of some prime factor ring of $R$. He also showed that any injective hull of a simple module is Artinian (see \cite{Matlis} and \cite[Proposition 3]{Matlis_DCC}). Vamos characterized commutative rings $R$ whose injective hulls of simples are Artinian as those whose localization $R_M$ by maximal ideals are Noetherian (\cite[Theorem 2]{Vamos}). Not necessarily commutative rings whose injective hulls of simples are Artinian were studied by Jans (\cite{Jans}) and termed {\it co-Noetherian}.
In \cite{Dahlberg} Dahlberg showed that injective hulls of simple modules over $U(\mathfrak{sl}_2)$ are locally Artinian. Since $U(\mathfrak{sl}_2)$ is an instance of a larger class of Noetherian domains, the Down-up Algebras, introduced by Benkart and Roby in \cite{BenkartRoby},
Patrick F. Smith asked which Noetherian Down-Up algebras satisfy this finiteness condition on their injective hulls of simple modules. The purpose of this note is to give a partial answer to Smith's question. Recall that a module is called {\it locally Artinian} if every of its finitely generated submodules is Artinian.

\subsection{}
In connection with the Jacobson Conjecture for Noetherian rings 
Jategaonkar showed in \cite{Jategaonkar_conj} (see also \cite{Cauchon, Schelter}) that the injective hulls of simple modules are locally Artinian provided the ring $R$ is fully bounded Noetherian.
\subsection{} Consider the following property for a ring $A$:
\begin{center}$(\diamond)\:\:$ Injective hulls of simple right $A$-modules are locally Artinian.\end{center} Property $(\diamond)$ is obviously equivalent to
 the condition, that all finitely generated essential extensions of simple right $A$-modules are Artinian. And in case $A$ is right Noetherian property $(\diamond)$ is further equivalent to the class of semi-Artinian right $A$-modules, i.e. modules $M$ that are the union of their socle series, to be closed under essential extensions. In torsion theoretic terms, $A$ has property $(\diamond)$ if and only if the hereditary torsion theory generated by all simple right $A$-modules is stable.

\subsection{} Let us first explain, why every commutative Noetherian ring has property $(\diamond)$ without using Matlis result. The Artin-Rees Lemma says (in one of its versions) that any ideal $I$ of a commutative Noetherian ring $A$ has the Artin-Rees property, i.e. for any essential extension $N\subseteq M$ of finitely generated $A$-modules with $NI=0$, there exists $n>0$ such that $MI^n=0$. Thus $M$ has a finite filtration $$0\subseteq MI^{n-1}\subseteq MI^{n-2}\subseteq \cdots\subseteq MI^2 \subseteq MI \subseteq M$$ such that each of its factors $MI^{k-1}/MI^k$ is a finitely generated $A/I$-module. If $N$ is a simple right $A$-module and $I=\Ann_A(N)$ then $A/I$ is a field, hence Artinian, and so any factor in the filtration of $M$ is Artinian, making $M$ an Artinian module.

A sufficient condition for $(\diamond)$ is therefore that right primitive ideals of $A$ have the Artin-Rees property and that primitive factor rings of $A$ are Artinian.

\subsection{}\label{SemiprimeKrullDimension1} For simple Noetherian algebras, the above argument cannot be used due to the absence of non-zero proper ideals. However if $A$ is a (not necessarily commutative) semiprime Noetherian ring of Krull dimension one, then for any essential right ideal $I$ of $A$, the Krull dimension of $A/I$ is one lower than the Krull dimension of $A$, hence Artinian.
For any extensions $E\subseteq M$ of a simple right $A$-module $E$ by a cyclic right $A$-module $M=A/I$, we have $E\simeq J/I$ with $J/I$ essential in $A/I$. Since pre-images of essential submodules are essential, also $J$ is essential in $A$. Thus $M/E\simeq A/J$ is Artinian and $M$ being an extension of the two Artinian modules $E$ and $M/E$ is also Artinian. Thus any semiprime Noetherian ring of Krull dimension one has property $(\diamond)$. This applies in particular to the first Weyl algebra $A_1(\CC)=\CC[x][y;\partial/\partial x]$.

\subsection{}\label{ReductionToAnnihilators} Let $E\subseteq M$ be an essential extension of a simple right $A$-module $E$ by a Noetherian module $M$. Let $P=\Ann_A(E)$ be its annihilator. Suppose there exists a non-zero central element $x\in P\cap Z(A)$, i.e. $Ex=0$. Denote by $f:M\rightarrow M$ the $A$-linear map $f(m)=mx$. Its kernel is $\Ker f=\Ann_M(x)$. By Fitting's Lemma there exists a number $n>0$ such that $\Im(f^n)\cap \Ker f^n=0$. Since $M$ is uniform and $E\subseteq \Ker f^n$ is non-zero, we have $\Im(f^n)=Mx^n=0$. Hence we have again a finite filtration $$0 \subseteq \Ker f=\Ann_M(x) \subseteq \Ker f^2 \subseteq \cdots \subseteq \Ker f^{n-1} \subseteq \Ker f^n=M$$ whose factors $\Ker f^k/\Ker f^{k-1}$ are $A/Ax$-modules and embed into $\Ann_M(x)$, because $f$ induces monomorphisms
$$M/\Ker f^{n-1} \hookrightarrow \Ker f^{n-1}/\Ker f^{n-2} \hookrightarrow \cdots \hookrightarrow \Ker f^2/\Ker f \hookrightarrow \Ker f.$$
Hence $M$ is Artinian if and only if $\Ann_M(x)=\Ker f$ is Artinian. The same argument also works for $x$ being a normal element. In this case $f$ is not $A$-linear anymore, but preserves submodules (see \cite[Lemma 2]{Jategaonkar}).

\subsection{} The last subsection allow us now to state the following reduction of our problem, in case $A$ has a non-trivial center.

\begin{prop}\label{Proposition_reduction_maximal} The following statements are equivalent for a countably generated Noetherian algebra $A$ with Noetherian center over an algebraically closed uncountable field $K$.
\begin{enumerate}
\item[(a)] Injective hulls of simple right $A$-modules are locally Artinian;
\item[(b)] Injective hulls of simple right $A/\mfm A$-modules are locally Artinian for all maximal ideals $\mfm$ of $Z(A)$.
\end{enumerate}
\end{prop}
\begin{proof}
$(a)\Rightarrow (b)$ is clear, since property $(\diamond)$ is inherited by factor rings.

$(b)\Rightarrow (a)$: First note that the Proposition is void if $A$ has trivial center $Z(A)=K$.
Hence we will suppose $Z(A)\neq K$. Moreover note, that any countably generated algebra $A$ over an uncountable field $K$ has the endomorphism property (\cite[9.1.8]{McConnellRobson}).
Hence the endomorphism ring of each simple right $A$-module $A$ is $\End{A}{E}\simeq K$ as $K$ was supposed to be algebraically closed.
Let $E$ be a simple right $A$-module and $M$ be a finitely generated essential extension of $E$.
Denote $P=\Ann_A(E)$ and $\mathfrak{m}=P\cap Z(A)$ which is a maximal ideal of $Z(A)$ as the $A$-action on $E$ restricts to an $Z(A)$-action on $E$ which is not faithful as $\End{A}{E}=K$ and $Z(A)\neq K$.
As $A$ and hence $Z(A)$ are Noetherian, there exist central elements $x_1, \ldots, x_k$ that generate $\mathfrak{m}$.
By \ref{ReductionToAnnihilators} $M$ is Artinian if and only if $M_1=\Ann_M(x_1)$ is Artinian.
Applying the same argument again leads to $M$ Artinian if and only if $M_2=\Ann_{M_1}(x_2)=\Ann_M(\{ x_1,x_2\})$ Artinian.
Iterating $k$ times leads to $M$ being Artinian if and only if $\Ann_M(\{x_1,\ldots, x_k\})=\Ann_M(\mfm)$ being Artinian.
Since $E\subseteq \Ann_M(\mathfrak{m})$ is an essential extension of $A/\mathfrak{m}A$-modules, with $\Ann_M(\mathfrak{m})$ being finitely generated, we get by hypothesis $(b)$, that $\Ann_M(\mathfrak{m})$ is Artinian.
\end{proof}

\subsection{}\label{Heisenberg} Let $\mfh$ be the $3$-dimensional Heisenberg algebra over $\CC$ which is generated by $x,y,z$ with Lie algebra structure given by $[x,y]=z$ and $[x,z]=0=[y,z]$. Let $A=U(\mfh)$. Then $Z(A)=\CC[z]$ and its maximal ideals are of the form $\mfm_\lambda=\langle z-\lambda\rangle$, with $\lambda\in \CC$. For $\lambda=0$, we have that $A/\mfm_0A \simeq \CC[x,y]$ is a commutative Noetherian domain and hence has property $(\diamond)$. For $\lambda \neq 0$, we have $A/\mfm_\lambda A \simeq \CC[x][y; \partial/\partial x]$ is the first Weyl algebra, which is a Noetherian domain of Krull dimension $1$ (see \cite[6.6.15]{McConnellRobson}) and hence also has property $(\diamond)$ by \ref{SemiprimeKrullDimension1}. Hence by Proposition \ref{Proposition_reduction_maximal} we have that $U(\mfh)$ has the property $(\diamond)$.

\subsection{}
In contrast to the Heisenberg algebra, which is a nilpotent Lie algebra, Ian Musson showed that no non-nilpotent soluble finite dimensional complex Lie algebra $\mfg$ has property $(\diamond)$, i.e. there exists a non-Artinian finitely generated essential extension of a simple right $U(\mfg)$-module (\cite{Musson_Counter}).

\subsection{}
In \cite{Dahlberg} Dahlberg showed that $U(\mathfrak{sl}_2)$ has property $(\diamond)$. Since $U(\mathfrak{sl}_2)$ and $U(\mfh)$ are two instances of a larger class of Noetherian domains, the Down-up Algebras, introduced by Benkart and Roby (\cite{BenkartRoby}) we ask which Noetherian Down-Up algebras satisfy property $(\diamond)$.

In the following section we will recall the definition of Down-up algebras and determine when they are fully bounded Noetherian. In the last section we show that some of the Noetherian Down-Up algebras of Krull dimension $2$ have property $(\diamond)$.
For simplicity all algebras are considered to be algebras over the complex numbers $\CC$.

\section{Fully bounded Noetherian Down-up algebras}

The down-up algebras form a three parameter family of associative algebras. For any parameter set $(\alpha,\beta, \gamma)\in\CC^3$ one defines a $\CC$-algebra, denoted by $A=A(\alpha,\beta,\gamma)$, generated by two elements $u$ and $d$ subject to the relations
\begin{eqnarray*}
d^2u &=& \alpha dud + \beta ud^2 + \gamma d\\
du^2 &=& \alpha udu + \beta u^2d + \gamma u
\end{eqnarray*}
\subsection{}
Kirkman, Musson and Passman proved that $A$ is noetherian if and only if it is a domain if and only if $\beta\neq 0$ if and only if $\CC[ud,du]$ is a polynomial ring (see \cite{KirkmanMussonPassman}).
\subsection{}\label{Kulkarni-gwa}
Any Noetherian down-up algebra can be represented as generalized Weyl algebra. Let $R$ be a commutative ring and $\sigma$ and automorphism of $R$ and $x$ an element of $R$, the generalized Weyl algebra is the algebra $R(\sigma,x)$ generated over $R$ in two indeterminates $u, d$ subject to the relations:
$ur=\sigma(r)u, dr=\sigma^{-1}(r)d $ for $r\in R$ and $ud=x, du=\sigma^{-1}(x)$.
In other words $$R(\sigma, x) := R\langle u, d\rangle / \langle ur-\sigma(r)u, dr-\sigma^{-1}(r)d, ud=x, du=\sigma^{-1}(x) \: \forall r\in R\rangle$$
As shown in \cite{KirkmanMussonPassman} and \cite{Kulkarni}, if $\beta\neq 0$, then $A\simeq R(\sigma, x)$ where $R=\CC[x,y]$,
$\sigma(x)=\frac{y-\alpha x - \gamma}{\beta}$ and $\sigma(y)=x$.
The isomorphism maps $ud$ to $x$ and $du$ to $y$. Kulkarni calls a Noetherian down-up algebra a down-up algebra at roots of unity if the associated automorphism $\sigma$ has finite order.

\subsection{}
The center of $R(\sigma, x)$ is generated by the fix ring $R^\sigma$ and the
elements $u^m, d^m$ where $m$ is the order of $\sigma$ or $0$ if the order is infinite (see \cite[2.0.1]{Kulkarni}).
Hence if $\sigma$ has finite order, then $R(\sigma, x)$ is generated by $1, u, \ldots, u^{m-1}, d, \ldots, d^{m-1}$ as a module over its center.
On the other hand if $\sigma$ has infinite order, then the center of $R(\sigma, x)$ is equal to the fix ring $R^\sigma$.
Hence $R(\sigma, x)$ cannot be finitely generated over a central subalgebra since otherwise it would be also finitely generated as a module over $R$ which is impossible since $R(\sigma, x) = \bigoplus_{n\in \ZZ} A_n$ is $\ZZ$-graded with $A_n=Ru^n$ and $A_{-n} = Rd^n$ for $n>0$ and $A_0=R$.
As any finitely generated $R$-submodule of $R(\sigma, x)$ is bounded and $A_n\neq 0$ for all $n$, $R(\sigma, x)$ is not finitely generated over $R$. Thus we proved:

\begin{lem}\label{module-finite-gwa}
$R(\sigma, x)$ is module-finite over a central subalgebra if and only if $\sigma$ has finite order.
\end{lem}

\subsection{}
Recall that a ring $R$ is called right (resp. left) bounded if every right (resp. left) essential ideal contains a non-zero two-sided ideal. $R$ is called right fully bounded noetherian if it is right Noetherian and every prime factor ring is right bounded. As mentioned in the first section, fully bounded Noetherian rings have property $(\diamond)$. The considerations above deduce now the following characterization of Down-up algebras at roots of unity.

\begin{thm}\label{FBN} The following statements are equivalent for a Noetherian Down-up algebra $A=A(\alpha,\beta,\gamma)$:
\begin{enumerate}
\item $A$ is a down-up algebra at roots of unity;
\item $A$ is module-finite over a central subalgebra;
\item $A$ satisfies a polynomial identity;
\item $A$ is fully bounded Noetherian;
\item The roots of the polynomial $X^2-\alpha X -\beta$ are distinct roots of unity such that both are also different from $1$ if $\gamma\neq 0$.
\end{enumerate}
\end{thm}

\begin{proof}
$(1)\Leftrightarrow (2)$ follows from Lemma \ref{module-finite-gwa}.
$(2)\Rightarrow (3)$ any module-finite algebra over a commutative subalgebra satisfies a polynomial identity (see for instance \cite[13.4.9]{McConnellRobson}).

$(3)\Rightarrow (4)$ any Noetherian algebra that satisfies a polynomial identity is fully bounded Noetherian (see for instance \cite[13.6.6]{McConnellRobson}).

$(4) \Rightarrow (1)$ we will show that $A$ is a down-up algebra at roots of unity.

Note that by \cite[p. 287]{CarvalhoMusson} any Noetherian Down-up algebra $A$ can be embedded into the skew Laurent polynomial ring $R[z,z^{-1}; \theta]$ where $R=\CC[x,y]$, $\theta(x)=y$ and $\theta(y)=\alpha y + \beta x + \gamma$. As a right $R$-module $R[z,z^{-1};\theta]$ is free on the basis $\{z^n \mid n\in \NN\}$ and the multiplication in $S$ is defined by $rz=z\theta(r)$ for $r\in R$. The embedding $\iota: A\rightarrow R[z,z^{-1};\theta]$ is given by $\iota(d)=z^{-1}$ and $\iota(u)=xz$, so that $\iota(ud)=x$ and $\iota(du)=y$.
By the proof of \cite[Lemma 1.2]{CarvalhoMusson}, $R[z,z^{-1};\theta]$ is the localization of $A$ by the Ore set $\{d^n \mid n\in \NN\}$.
By \cite[4.1.8]{PaulaThesis} (or by \cite[Proposition 1.5]{Cauchon} and \cite[Theorem 3.5]{Krause}), if $A$ is fully bounded Noetherian, then also $A_d$, hence $R[z,z^{-1};\theta]$. By \cite[Proposition 4.1.12]{PaulaThesis}, $\theta$ must have finite order. Since $A\simeq A'=A(-\alpha\beta^{-1}, \beta^{-1}, -\gamma\beta^{-1})$ by \cite[Lemma 4.1]{CarvalhoMusson}, also $\theta'$ has finite order, where $\theta'$ is the automorphism of $\CC[x,y]$ defined analogously by $\theta'(x)=y$ and $\theta'(y)=(-\alpha\beta^{-1})y + (-\beta^{-1})x + (-\gamma\beta^{-1})$. Denoting by $\tau$ the automorphism of $\CC[x,y]$ which interchanges $x$ and $y$, we have that $\sigma=\tau\theta'\tau$ has finite order, where $\sigma$ equals the automorphism that represents $A$ as a generalized Weyl algebra as in \ref{Kulkarni-gwa}. Hence $A$ is a down-up algebra at roots of unity.

$(4)\Leftrightarrow (5)$ Let $r_1, r_2$ be the roots of the polynomial $X^2-\alpha X - \beta$ and let $\theta$ be the automorphism that defines the skew Laurent ring $R[z,z^{-1}; \theta]$ as above, with $R=\CC[x,y]$. Note that $\theta$ stabilizes the vector space $V$ spanned by $1, x$ and $y$. In \cite[p.288-289]{CarvalhoMusson} a basis $1, w_1, w_2$ for $V$ had been found such that the matrix of $\theta$ with respect to this basis is in Jordan canonical form. Four cases had to be considered: if both roots $r_1$ and $r_2$ are different and also different from $1$, then there exists such a basis such that $\theta(w_i)=r_iw_i$ for $i=1,2$. Hence $\theta$ has finite order if and only if both roots are roots of unity.

If $r_1=1$ and $r_2\neq 1$, then there exists a basis such that $\theta(w_1)=w_1+\gamma$ and $\theta(w_2)=r_2w_2$. Hence $\theta$ has finite order if and only if $\gamma=0$ and $r_2$ is a root of unity.

If both roots are the same $r=r_1=r_2$ but different from $1$, then there exists a basis such that $\theta(w_1)=rw_1$ and
$\theta(w_2)=rw_2+w_1$. Hence for any $n$, $\theta^n(w_1)=r^nw_1$ and $\theta^n(w_2)=r^nw_2 + nr^{n-1}w_1$. Hence $\theta$ cannot have finite order. Similarly, if both roots are $1$, there exists a basis such that $\theta(w_1)=w_1+\gamma$ and $\theta(w_2)=w_2+w_1$ that implies that $\theta$ will not have finite order.
\end{proof}

\section{Non-primitive Down-Up algebras of Krull dimension two}

A theorem of Bavula and Lenagan states, that the Krull dimension of $A=A(\alpha,\beta,\gamma)$ is $2$ if and only if $\alpha+\beta=1$ and $\gamma\neq 0 \neq \beta$; otherwise the Krull dimension is $3$ (see \cite[Theorem 4.2]{BavulaLenagan}).
Equivalently $A$ has Krull dimension $2$ precisely if $\gamma,\beta \neq 0$ and $1$ is a root of $X^2-\alpha X-\beta$. We will focus in this section on those Down-Up algebras which are denoted by $A_\eta$ in \cite{CarvalhoMusson}: $A_\eta := A(1+\eta, -\eta, 1)$ for $\eta\in \CC^\times=:\CC\setminus\{0\}$.

\subsection{}
By \cite[Theorem 1.3(g)]{Zhao} the center of the algebras $A_\eta$, with $\eta$ being a primitive N-th root of unity different from $1$, is a polynomial ring $\CC[\omega]$ in one variable, where $\omega$ is the element $\omega=z^N$ with $z=du-ud+\frac{\gamma}{\eta-1}$. Note that if $\eta$ is not a root of unity, then the center of $A_\eta$ is trivial. We will apply Proposition \ref{Proposition_reduction_maximal} to prove the following:

\begin{thm}\label{A_eta} $A_\eta$ satisfies $(\diamond)$ if $\eta$ is a root of unity different from $1$.
\end{thm}

\begin{proof} Let $\eta$ be a primitive $N$th root of unity different from $1$. We intend to use Proposition \ref{Proposition_reduction_maximal}.
As mentioned before $Z(A_\eta)=\CC[\omega]$. The maximal ideals of $\CC[\omega]$ are of the form $\langle \omega-c\rangle$ with $c\in \CC$. By \cite[Theorem 8.1(C1)]{Praton} any ideal of the form $(\omega-c)A_\eta$ with $c\in\CC^\times$ is right primitive.
Hence $A_\eta/(\omega-c)A_\eta$ is a right primitive Noetherian ring of Krull dimension $1$ and hence has property $(\diamond)$ by \ref{SemiprimeKrullDimension1}.
For $c=0$, let $B=A_\eta/\omega A_\eta$. We have that $\omega=z^N$ with $z=du-ud+\frac{\gamma}{\eta-1}$. As $z$ is a normal element of $A_\eta$ it is also normal in $B$. By \cite[Theorem 8.1(C1)]{Praton} $zA_\eta$ is a primitive right ideal of $A_\eta$ and hence also of $B$. Thus $B/zB$ is a primitive Noetherian ring of Krull dimension $1$ and has property $(\diamond)$ again by \ref{SemiprimeKrullDimension1}. Given any essential extension $E\subseteq M$ of finitely generated $B$-modules, with $E$ being simple, we first note, that $zE=0$, since otherwise $E=zE=\cdots = z^NE = \omega E = 0$ - a contradiction. Since $B/zB$ satisfies $(\diamond)$, $\Ann_M(z)$ is Artinian and by \cite[Lemma 2]{Jategaonkar} $M$ is Artinian.

This shows that any factor $A_\eta/\mfm A_\eta$ by a maximal ideal $\mfm$ of $Z(A_\eta)$ has property $(\diamond)$. By Proposition \ref{Proposition_reduction_maximal} $A_\eta$ satisfies $(\diamond)$.
\end{proof}

 \subsection{}
 Summarizing Theorem \ref{FBN} and Theorem \ref{A_eta} we have the following:

 \begin{cor} The injective hull of any simple right $A(\alpha,\beta,\gamma)$-module is locally Artinian, if the roots of $X^2-\alpha X - \beta$ are distinct roots of unity or both equal to one.
 \end{cor}

 \begin{proof}
 If the roots of $X^2-\alpha X - \beta$ are distinct roots of unity and also different from $1$ if $\gamma\neq 0$, then $A=A(\alpha,\beta,\gamma)$ is fully bounded Noetherian by Theorem \ref{FBN}. A classical result by Schelter and Jategaonkar says that the injective hull of a simple right $R$-module over a left Noetherian right fully bounded Noetherian ring $R$ is locally Artinian (see for instance \cite[9.12]{GoodearlWarfield} or \cite[Proposition 6.4.14]{McConnellRobson}).

 Suppose $\gamma\neq 0$ and that one of the roots is $1$, then $A\simeq A_\eta$ and Theorem \ref{A_eta} shows that $A$ has property $(\diamond)$.
 In case both roots are $1$, then $\alpha=2$ and $\beta=-1$. Since $A_1=A(2,-1,1)=U(\mathfrak{sl}_2)$ and $A(2,-1,0)=U(\mfh)$ those algebras have property $(\diamond)$ by \cite{Dahlberg} and \ref{Heisenberg}.
 \end{proof}

\subsection{}
We were unable to find an example of a Noetherian Down-Up algebra that does not satisfy $(\diamond)$.

\end{document}